\documentclass[a4paper,10pt]{article}
\usepackage[utf8]{inputenc}
\usepackage{ucs}
\usepackage{amsmath}
\usepackage{amsfonts}
\usepackage{amssymb}
\usepackage{amsthm}
\usepackage[english]{babel}
\usepackage{fontenc}
\usepackage{graphicx}
\usepackage{geometry}
\usepackage{hyperref}
\usepackage{tikz}
\usepackage{tkz-euclide}
\usepackage[font=small,labelfont=sc]{caption}
\usepackage{fancyhdr}

\fancypagestyle{specialfooter}{%
  \fancyhf{}
   
  \fancyfoot[L]{\small\quad This work was supported by the UK Engineering and Physical Sciences Research Council (EPSRC) studentship award 1789662.}
}
\newtheorem{theorem}{Theorem}
\newtheorem*{theorem*}{Theorem}
\newtheorem{lemma}{Lemma}
\newtheorem*{lemma*}{Lemma}
\newtheorem{definition}{Definition}
\newtheorem{proposition}{Proposition}
\newtheorem*{proposition*}{Proposition}

\title{Pfaffian control of some polynomials involving the $j$--function and Weierstrass elliptic functions}
\author{John Armitage}

\begin{document}

\maketitle


\section{Introduction}
\thispagestyle{specialfooter}

In this paper, we obtain some new bounds on the number of zeros of polynomials in $z$ and $j(z)$, and polynomials in $z$ and $\wp(z)$, where $\wp$ is a Weierstrass $\wp$--function associated to a lattice of the form $\langle1,i\tau\rangle$, where $\tau$ is real. By the argument principle, the zeros of a holomorphic function in a closed, compact, simply connected region of the complex plane are controlled by the winding number of the function on its boundary -- if the function is tame on the boundary, then we will obtain control over the number of zeros within the region. For $j$, and $\wp$, we use Pfaffian definitions of their inverses on suitable contours, and results of Khovanskii \cite{fewnomials} to bound the zeros.

For polynomials in $z$ and $j(z)$, we obtain the following, 
  \begin{theorem}
   Let $P(X,Y)$ be a complex polynomial of degree at most $d$ in either variable. Then $P(z,j(z))$ has at most $2^{68}d^{10}$ zeros in the standard fundamental domain.
  \end{theorem}
  This we believe is a new result, giving a bound on the whole (non--compact) fundamental domain. It may be compared to Binyamini's result of a similar nature for Noetherian functions on relatively compact domains \cite{Binyamini}. An application of this result here would yield a bound which depends on the size of a domain in which the zeros lie, and the size of $\lvert j\rvert$ on it. Note also that $j$ is o--minimal, which gives an ineffective finiteness result. The zero--bound depends polynomially on degree,  and in that sense is not too far from the truth, as an obvious lower bound is $\gg d^2$. It is known that the inverse of $j(ix)$ is real and Pfaffian on the imaginary axis, a fact which may be deduced from its expression in terms of the Gaussian hypergeometric functions, which are themselves Pfaffian on an interval, see for example \cite{Jones} and \cite{Bianconi}, who make use of this in the expression of elliptic functions. We make use of a result from Ramanujan's theory of elliptic functions to alternative bases to obtain a direct expression for the inverse in terms of Gaussian hypergeometric functions. Near the cusps it approximates to the $q^{-1}$ term in its $q$-expansion. Though $j$ is a Noetherian function, it is unknown whether its real and imaginary parts are Pfaffian in the whole fundamental domain, which would directly furnish a zero--bound, so considering a contour containing the fundamental domain avoids this. 
  
  For the Weierstrass $\wp$ functions, we restrict our attention to those with square  lattice, as in this case $\wp$ is real on the boundary of any fundamental domain, which improves the estimates under consideration, and obtain the following,
   \begin{theorem}
  Suppose that $\wp(z)$ is the Weierstrass $\wp$ function associated to the lattice $\langle1,i\tau\rangle$, where $\tau>0$ is real, and let $P(X,Y)$ be a complex polynomial of degree $d$ in either variable. Then  $P(z,\wp(z))$ has at most $8d^2+14d+5$ zeros in each fundamental domain.
 \end{theorem}
 This may be compared with the non--uniform bound $c(\wp)d^2$ (see for example \cite{masser}), which holds for all $\wp$. A similar result is that of Jones and Schmidt \cite{Jones} -- they give a Pfaffian definition of $\wp$ on the whole fundamental domain, which yields a uniform bound of $cd^{10}$, where $c$ is an absolute constant. For non--real invariants, a similar result by our method of inferior quality could be obtained, as the need to take real and imaginary parts complicates the Pfaffian definitions involved.
 
\section{Preliminaries}

We first have two lemmas bounding the arc integral related to the argument principle, which we will apply to segments of a closed contour, the first to estimate where there is a dominant term, and the second in order to apply bounds on zeros of real functions.
\begin{lemma}\label{integralremainder}
 Suppose that $\lvert f(z)\rvert >C\lvert g(z)\rvert$ for some $C>1$ on the contour $\Gamma$, and $f(z),g(z)$ are non--zero on $\Gamma$. Then
 \begin{equation*}
  \left\lvert\int_\Gamma\frac{(f+g)'(z)}{(f+g)(z)}dz\right\rvert\leq\left\lvert\int_\Gamma\frac{f'(z)}{f(z)}dz\right\rvert+\frac{C}{C-1}\lvert\Gamma\rvert\sup_{z\in\Gamma}\left\{\frac{\lvert f'(z)\rvert\lvert g(z)\rvert}{\lvert f(z)\rvert^2}+\frac{\lvert g'(z)\rvert}{\lvert f(z)\rvert}\right\},
 \end{equation*}
 where $\lvert\Gamma\rvert$ is the length of $\Gamma$.
\end{lemma}
\begin{proof}
 Considering the difference, we have
 \begin{align*}
  \left\lvert\int_\Gamma\left(\frac{f'(z)}{f(z)}-\frac{f'(z)+g'(z)}{f(z)+g(z)}\right)dz\right\rvert&=\left\lvert\int_\Gamma\frac{f'(z)g(z)+g'(z)f(z)}{f(z)(f(z)+g(z))}dz\right\rvert\\
  &\leq\int_\Gamma\frac{\lvert f'(z)g(z)\rvert+\lvert g'(z)f(z)\rvert}{\lvert f(z)\rvert(\lvert f(z)\rvert-\lvert g(z)\rvert)}dz\\
  &\leq\frac{\lvert\Gamma\rvert}{1-\frac{1}{C}}\sup_{z\in\Gamma}\left\{\frac{\lvert f'(z)\rvert\lvert g(z)\rvert}{\lvert f(z)\rvert^2}+\frac{\lvert g'(z)\rvert}{\lvert f(z)\rvert}\right\}.
 \end{align*}
\end{proof}
\begin{lemma}\label{integralzeros}
 \begin{equation*}
    \frac{1}{2\pi}\left\lvert\int_\Gamma\frac{f'(z)}{f(z)}dz\right\rvert\leq\#\{\text{\textnormal Im}(f(z))=0|z\in\Gamma\}/2+1
 \end{equation*}
 and
 \begin{equation*}
    \frac{1}{2\pi}\left\lvert\int_\Gamma\frac{f'(z)}{f(z)}dz\right\rvert\leq\#\{\text{\textnormal Re}(f(z))=0|z\in\Gamma\}/2+1.
 \end{equation*}
\end{lemma}

The following definition is less general that that of Pfaffian functions in \cite{fewnomials}, but is easier to work with (and restricted to one dimension),
\begin{definition}
 Let $f_1,\ldots,f_r$ be a sequence of analytic functions on the interval $(a,b)$. Then $f_1,\ldots,f_r$ is a \emph{Pfaffian chain} of degree $\alpha$ if, for $1\leq i\leq r$ and $x\in(a,b)$,
 \begin{equation*}
  \frac{df_i(x)}{dx}=P(x,f_1(x),\ldots,f_i(x)),
 \end{equation*}
 and the maximum total degree of each $P_i$ is $\alpha$.
 A \emph{Pfaffian function} of order $r$ and degree $(\alpha,\beta)$ is a function $P(x,f_1(x),\ldots,f_s(x))$ where $P$ is a polynomial of total degree at most $\beta$, and $f_1(x),\ldots,f_s(x)$ are members of a Pfaffian chain of order $r$ and degree $\alpha$.
\end{definition}
We make use of the following bound on the number of zeros of a Pfaffian function,
\begin{theorem}[\cite{fewnomials}]\label{MainBound}
 Let $f$ be a Pfaffian function of order $r$ and degree $(\alpha,\beta)$ on the open interval $(a,b)$. Then the number of zeros of $f$ in $(a,b)$ is at most
 \begin{equation*}
  2^{r(r-1)/2}\beta(\alpha+\beta)^r.
 \end{equation*}
\end{theorem}

\section{Polynomials in $z$ and $j(z)$}

Here we make use of an expression for the inverse of $j$ in terms of the Gaussian hypergeometric function, which is given, for $\lvert z\rvert<1$, by
\begin{equation*}
_2F_1(a,b,c;z) = \sum_{n=0}^\infty\frac{(a)_n(b)_n}{(c)_n}\frac{z^n}{n!},
\end{equation*}
where $(a)_n=\prod_{k=0}^{n-1}(a+k)$ is the rising factorial, and $c$ not a non--positive integer. Gauss determined 15 relations between those $_2F_1$ whose parameters differ by integers. We make use of the following two relations, where we have suppressed the parameters of the function, and use the notation $F(c\pm)=$$_2F_1(a,b,c\pm1;z)$,
\begin{theorem}[Gauss, \cite{gauss}, Art. 11, Eq. 16; DLMF \cite{DLMF}, Eq. 15.5.21]
 \begin{align*}
  z\frac{dF}{dz}&=(c-1)F(c-)-F)\\
  &=z\frac{(c-a)(c-b)F(c+)+c(a+b-c)F}{c(1-z)}.
 \end{align*}

\end{theorem}

Now these allow us to construct a Pfaffian chain which $j^{-1}$ will be defined over on the imaginary axis. 
\begin{lemma}
 The sequence of functions
 \begin{align*}
  \frac{1}{x}, && \frac{1}{1-x}, && \frac{F}{F(c+)}, && F(c+), && F, && \frac{1}{F}
 \end{align*}
 is a Pfaffian chain of degree $3$.
\end{lemma}
\begin{proof}

We first note that $\frac{1}{x}$ and $\frac{1}{1-x}$ are Pfaffian of order $1$ and degree $(2,1)$. By Gauss' contiguous relations, we have
\begin{align*}
 x\frac{dF(c+)}{dx} &= c(F-F(c+))\\
 \frac{dF}{dx} &= \frac{(c-a)(c-b)F(c+)+c(a+b-c)F}{c(1-x)}.
\end{align*}
Consider 
 \begin{align*}
  \frac{d}{dx}\frac{F}{F(c+)} &= \frac{(c-a)(c-b)F(c+)+c(a+b-c)F}{c(1-x)F(c+)} - \frac{c(F-F(c+))F}{xF(c+)^2}\\
  &=\frac{(c-a)(c-b)}{c(1-x)}+\left(\frac{a+b-c}{c(1-x)}-\frac{c-1}{x}\right)\frac{F}{F(c+)}-\frac{c}{x}\left(\frac{F}{F(c+)}\right)^2.
 \end{align*}
 Hence $\frac{F}{F(c+)}$ is a Pfaffian function of order $3$ and degree $(2,1)$. Now consider
 \begin{equation*}
 \frac{dF(c+)}{dx} = \frac{c(F-F(c+))}{x} = \frac{c}{x}\left(F(c+)\frac{F}{F(c+)} - F(c+)\right),
\end{equation*}
so $F(c+)$ is a Pfaffian function of order $4$ and degree $(3,1)$. Hence $F=F(c+)\frac{F}{F(c+)}$ is a Pfaffian function of order 5 and degree $(3,2)$. Further, $1/F$ is Pfaffian of order $6$ and degree $(3,1)$.
\end{proof}

We make use of the following consequence, which is clear considering the above argument for $_2F_1\left(\frac{1}{6},\frac{5}{6},1,\frac{1}{2}+\frac{y}{2}\right)$ and $ _2F_1\left(\frac{1}{6},\frac{5}{6},1,\frac{1}{2}-\frac{y}{2}\right)$.
\begin{lemma}\label{PfaffOrder}
 $\frac{_2F_1\left(\frac{1}{6},\frac{5}{6},1,\frac{1}{2}+\frac{y}{2}\right)}{_2F_1\left(\frac{1}{6},\frac{5}{6},1,\frac{1}{2}-\frac{y}{2}\right)}$ is Pfaffian function of order $9$ and degree $(2,1)$ on $(0,1)$.
\end{lemma}

Now we express the inverse of $j$ on the imaginary axis in terms of the hypergeometric functions by way of the following theorem, an inversion formula from the theory of elliptic functions to alternative bases -- in this case the sextic theory.

\begin{theorem}[Theorem 4.10, $r=1$, \cite{ram}]\label{rambase}
 Let $q$ be a real number in the interval $0<q<1$. Then
 \begin{equation*}
  q=\exp\left(-2\pi\frac{_2F_1\left(\frac{1}{6},\frac{5}{6},1,1-x\right)}{_2F_1\left(\frac{1}{6},\frac{5}{6},1,x\right)}\right),
 \end{equation*}
 where, letting $P,Q,R$ be Ramanujan's Eisenstein series,
 \begin{equation*}
  x(1-x)=\frac{Q(q)^3-R(q)^2}{4Q(q)^3}.
 \end{equation*}
\end{theorem}
Letting $q=e^{2\pi i\tau}$, and $\tau$ be purely imaginary, we take
\begin{equation*}
 \alpha = \frac{1}{2} - \frac{1}{2}\sqrt{1-\frac{1728}{j(\tau)}},
\end{equation*}
which satisfies
\begin{equation*}
 \alpha(1-\alpha)=\frac{1728}{4j(\tau)}=\frac{Q(q(\tau))^3-R(q(\tau))^2}{4Q(q(\tau))^3},
\end{equation*}
so that by Theorem \ref{rambase},
\begin{equation*}
 \tau =i\frac{_2F_1\left(\frac{1}{6},\frac{5}{6},1,\frac{1}{2} + \frac{1}{2}\sqrt{1-\frac{1728}{j(\tau)}}\right)}{_2F_1\left(\frac{1}{6},\frac{5}{6},1,\frac{1}{2} - \frac{1}{2}\sqrt{1-\frac{1728}{j(\tau)}}\right)}.
\end{equation*}
So for $x \geq 1728$, letting $J(x)=j(ix)$,
  \begin{equation*}
   J^{-1}(x) = \frac{_2F_1\left(\frac{1}{6},\frac{5}{6},1,\frac{1}{2} + \frac{1}{2}\sqrt{1-\frac{1728}{x}}\right)}{_2F_1\left(\frac{1}{6},\frac{5}{6},1,\frac{1}{2} - \frac{1}{2}\sqrt{1-\frac{1728}{x}}\right)}.
  \end{equation*}

\begin{figure}
\centering
   \begin{tikzpicture}[scale=3]
      \draw[thick] (-2,0) -- (2,0);
      \draw[thick] (0,0) -- (0,3);
      
      \draw[dashed] (0.5,0) -- (0.5,3);
      \draw[dashed] (-1.5,0) -- (-1.5,3);
      \draw[dashed] (1.5,0) -- (1.5,3);
      \draw[dashed] (-0.5,0) -- (-0.5,3);
      
      \draw[dashed] (1,0) arc (0:180:1);
      \draw[dashed] (2,0) arc (0:180:1);
      \draw[dashed] (0,0) arc (0:180:1);
      \draw[dashed] (-1,0) arc (0:90:1);
      \draw[dashed] (2,1) arc (90:180:1);
      \draw[dashed] (0,0) arc (0:180:1/3);
      \draw[dashed] (1,0) arc (0:180:1/3);
      \draw[dashed] (2/3,0) arc (0:180:1/3);
      \draw[dashed] (-1/3,0) arc (0:180:1/3);
      
      \draw[ultra thick][->] (-0.5,2)--(0,2);
      \draw[ultra thick][->] (-1,1/2)--(-1,1.25);
      \draw[ultra thick][->] (1,2)--(1,1.25);
      
      \draw[ultra thick][->] (0.5001,0.50)--(0.5,0.5);
      \draw[ultra thick][->] (-0.5,0.50)--(-0.5001,0.5);

      \draw[ultra thick] (-0.99,0.5) -- (-1,0.5) -- (-1,2) -- (1,2) -- (1,0.5) -- (0.99,0.5);
      \draw [ultra thick,  domain=-0.5:0.5, samples=40] 
      plot ({\x/(\x*\x+4)}, {2/(\x*\x+4)} );
      \draw [ultra thick,  domain=0:0.5, samples=40] 
      plot ({\x/(\x*\x+4)-1}, {2/(\x*\x+4)} );
      \draw [ultra thick,  domain=-0.5:0, samples=40] 
      plot ({\x/(\x*\x+4)+1}, {2/(\x*\x+4)} );
      \draw [ultra thick,  domain=0.5:2, samples=40] 
      plot ({\x*\x/(\x*\x+1)}, {\x/(\x*\x+1)} );
      \draw [ultra thick,  domain=0.5:2, samples=40] 
      plot ({\x*\x/(\x*\x+1)-1}, {\x/(\x*\x+1)} );
      \draw [ultra thick,  domain=-0.5:0, samples=40] 
      plot ({(\x*\x+\x+4)/((\x+1)*(\x+1)+4)}, {2/((\x+1)*(\x+1)+4)} );
      \draw [ultra thick,  domain=-0.5:0, samples=40] 
      plot ({-(\x*\x+\x+4)/((\x+1)*(\x+1)+4)}, {2/((\x+1)*(\x+1)+4)} );
      \draw [ultra thick,  domain=-0.5:0, samples=40] 
      plot ({-(\x*\x+\x+4)/((\x+1)*(\x+1)+4)+1}, {2/((\x+1)*(\x+1)+4)} );
      \draw [ultra thick,  domain=-0.5:0, samples=40] 
      plot ({(\x*\x+\x+4)/((\x+1)*(\x+1)+4)-1}, {2/((\x+1)*(\x+1)+4)} );
      \draw[ultra thick] (0.192,0.4105) -- (0.2,0.4) -- (0.205,0.404);
      \draw[ultra thick] (-0.192,0.4105) -- (-0.2,0.4) -- (-0.205,0.404);
      \draw[ultra thick] (1-0.192,0.4105) -- (1-0.2,0.4) -- (1-0.205,0.404);
      \draw[ultra thick] (0.192-1,0.4105) -- (0.2-1,0.4) -- (0.205-1,0.404);

   \end{tikzpicture}
   
   \caption{}
   \label{contour}
   
\end{figure}

\begin{proof}[Proof of Theorem 1]
  
  We apply the argument principle to a truncated fundamental domain together with its copies or half--copies under SL$_2(\mathbb{Z})$ indicated by Figure \ref{contour}. First consider a contour $\Gamma$ within the one indicated, containing all zeros within the original contour, such that $P(z,j(z))$ is non-zero on this contour. Let $0<\epsilon<\min_{z\in\Gamma}\lvert P(z,j(z))\rvert$. Then by Rouch\'{e}'s theorem, the number of zeros of $P(z,j(z))+\epsilon e^{i\theta}$ within $\Gamma$ is equal to that of $P(z,j(z))$, for any $\theta$. Choose $\theta$ so that $P_\epsilon(z,j(z)):=P(z,j(z))+\epsilon e^{i\theta}$ is non--zero on the original contour -- this is possible by discreteness of zeros of analytic functions.

  We now bound the winding number of $P_\epsilon(z,j(z))$ on this boundary. The contour is chosen to be dominated by the $q^{-1}$ term of $j$ near the cusps. Let $j(it) = J(t)$. We will refer to elements of SL$_2(\mathbb{Z})$ by $g$, with entries
  \[\begin{pmatrix}
     \tilde{a}&\tilde{b}\\
     \tilde{c}&\tilde{d}
    \end{pmatrix}
\]
  acting on $z$ by
  \begin{equation*}
   g(z)=\frac{\tilde{a}z+\tilde{b}}{\tilde{c}z+\tilde{d}}.
  \end{equation*}
  The contour is composed of lines or curves near the cusps, and images of the imaginary axis under the action of some $g$.
  
  \paragraph{Copies of the imaginary axis $i\mathbb{R}_{\geq1}$:} Here we have $\text{Im}(P\left(g(it),j\left(g(it)\right)\right))  = \text{Im}\left(P\left(\frac{\tilde{a}it+\tilde{b}}{\tilde{c}it+\tilde{d}},j\left(it\right)\right)\right)= \text{Im}\left(P\left(\frac{(\tilde{a}it+\tilde{b})(-\tilde{c}it+\tilde{d})}{\tilde{c}^2t^2+\tilde{d}^2},j\left(it\right)\right)\right) = \left(\frac{1}{\tilde{c}^2t^2+\tilde{d}^2}\right)^dQ(t,J(t))$ for some real $Q$, as $j(it)$ is real. $Q$ is of degree $\leq2d$, and has the same number of zeros as $\text{Im}(P\left(g(it),j\left(g(it)\right)\right))$.
  
  As $J(t)$ is $1$--$1$ from $(1,\infty)$ to $(1728,\infty)$, letting $x = J(t)$, the number of zeros of $Q(t,J(t))$ is equal to that of $Q(J^{-1}(x),x)$ on $x>1728$, and as $\sqrt{1-\frac{1728}{x}}$ is $1$--$1$ on $x>1728$, letting $y=\sqrt{1-\frac{1728}{x}}$, the number of zeros of $Q(J^{-1}(x),x)$ is equal to that of 
  \begin{equation*}
   Q\left(J^{-1}\left(\frac{1728}{1-y^2}\right),\frac{1728}{1-y^2}\right)=\left(\frac{1}{1-y^2}\right)^{2d}R\left(J^{-1}\left(\frac{1728}{1-y^2}\right),y\right)
  \end{equation*}
  on $(0,1)$, where $R$ is of degree $\leq4d$. Now using the expression for $J^{-1}$ in terms of the hypergeometric series $_2F_1$, we have that the number of zeros of Im$(P(g(it),j(g(it))))$ is bounded by that of
  \begin{equation*}
   R\left(\frac{_2F_1\left(\frac{1}{6},\frac{5}{6},1,\frac{1}{2} + \frac{y}{2}\right)}{_2F_1\left(\frac{1}{6},\frac{5}{6},1,\frac{1}{2} - \frac{y}{2}\right)},y\right)
  \end{equation*}
  for $y\in(0,1)$. Now by Lemma \ref{PfaffOrder}, this is a Pfaffian function of order $9$ and degree $(3,4d)$. By Theorem \ref{MainBound}, the number of zeros is then bounded by $2^{64}d^{10}$, so by Lemma \ref{integralzeros}, on this copy of the imaginary axis, we have
  \begin{equation*}
   \frac{1}{2\pi}\left\lvert\int_\gamma\frac{d}{dz}(P(z,j(z)))\frac{1}{P(z,j(z))}dz\right\rvert\leq 2^{64}d^{10}.
  \end{equation*}
  
  \paragraph{Copies of $\text{\normalfont Im}(z) = Y$:}
  
  For these sections of the boundary, we have
  \begin{equation*}
   P(g(x+iY),j(g(x+iY))) = \left(\frac{1}{ \tilde{c}(x+iY)+\tilde{d}}\right)^dQ(x+iY,j(x+iY)),
  \end{equation*}
  so that 
  \begin{equation*}
   \int_\gamma\frac{\frac{d}{dz}(P(g(z),j(z)))}{P(g(z),j(z))}dz =  \int_\gamma\frac{\frac{d}{dz}(Q(z,j(z)))}{Q(z,j(z))}-\frac{d\tilde{c}}{\tilde{c}z+\tilde{d}}dz.
  \end{equation*}
  For sufficiently large $Y$, $\frac{d\tilde{c}}{\tilde{c}z+\tilde{d}}<0.01$, so we consider the remaining term in $Q$.
  
  Letting $j(z)^lh(z)$ be the term with the largest power of $j$ occurring in $Q(z,j(z))$, where $h(z)$ is its coefficient over the polynomials in $z$, let the degree of $h$ be $n$. For sufficiently large $Y$, $j(x+iY)(iY)^n$ is the dominant term in $Q(x+iY,j(x+iY))$, i.e. $\lvert j(x+iY)(iY)^n\rvert>2\lvert Q(x+iY,j(x+iY))-j(x+iY)(iY)^n\rvert$. In the same way the $e^{-2\pi i(x+iY)}$ term in the $q$-expansion of $j$ dominates $j(x+iY)$ when $Y$ is large, and so as $Y\to\infty$,
  \begin{equation*}
   \frac{\left\lvert Q(x+iY,j(x+iY))-e^{-2l\pi i(x+iY)}(iY)^n\right\rvert}{\left\lvert e^{-2l\pi i(x+iY)}(iY)^n\right\rvert}\to 0,
  \end{equation*}
  and the same holds for the derivative of the numerator.
  So taking $f(z)=(iY)^ne^{-2l\pi i z}$ and $g(z) = Q(z,j(z))-f(z)$, for sufficiently large $Y$, we may take $C=2$ in Lemma \ref{integralremainder}, and obtain the bound
  \begin{align*}
   \left\lvert\int_\gamma\frac{\frac{d}{dz}(Q(z,j(z)))}{Q(z,j(z))}dz\right\rvert&\leq\left\lvert\int_\gamma\frac{\frac{d}{dz}(e^{-2l\pi iz}(iY)^n)}{e^{-2l\pi iz}(iY)^n}dz\right\rvert+2\sup_{z\in\gamma}\left\{\frac{\lvert f'(z)\rvert\lvert g(z)\rvert}{\lvert f(z)\rvert^2}+\frac{\lvert g'(z)\rvert}{\lvert f(z)\rvert}\right\}\\
   &\leq 2\pi l+0.01\\
   &\leq 2\pi d+0.01.
  \end{align*}
  We also take $Y$ sufficiently large to ensure all zeros in the fundamental domain (and its copies) are interior to the contour, which is possible by virtue of the dominating term in $P(z,j(z))$ at the various cusps (or by the o--minimality of $j$ implying there are only finitely many). Finally, the integral over the entire contour is bounded in absolute value by $8\cdot2^{64}d^{10}+10d+0.2\leq2^{68}d^{10}$, as there are $8$ copies of the imaginary axis, and $10$ copies or half--copies of the line $(-1/2+iY,1/2+iY)$.
 
 \end{proof}

 \section{Polynomials in $z$ and the Weierstrass $\wp$--function}
 
  Here we make use of the following theorem of Khovanskii,
 
\begin{theorem}[\S2.3, Theorem 2, \cite{fewnomials}]\label{preimages}
Let $G : R^{n+1} \to R^{1}$ be a smooth function with nondegenerate level set
$M^n$. Let $F : R^{n+1} \to R^n$ be a smooth proper map, and $\tilde{F} : M^n \to R^n$ its restriction to $M^n$. Let, further, $\hat{J}$ be any smooth function on $R^n$ that coincides on $M^n$ with the Jacobian $J$ of the map $(F, G) : R^{n+1} \to R^n \times R^1$. Under these conditions the following holds: the maximum number of nondegenerate preimages of any point in the range of of the map $\tilde{F} : M^n \to R^n$ is bounded by that of the map $(F, \hat{J}) : R^{n+1}\to R^n \times R^1$.
\end{theorem}
We apply this to a polynomial in $x$ and $\wp^{-1}(x)$, where $\wp$ is a Weierstrass $\wp$--function, and use the argument principle to bound the zeros of $P(z,\wp(z))$ in its fundamental domain. We consider $\wp$ with lattices of the form $\langle1,i\tau\rangle$, $\tau>0$ real. The bound on the number of zeros follows in a similar way to \S 2.3 Theorem 1 of \cite{fewnomials} -- we proceed in this manner to give a better bound than simply applying Theorem \ref{MainBound}.

First we have the differential equation satisfied by $\wp$,
\begin{equation*}
 \wp'(z)^2=4\wp(z)^3-g_2\wp(z)-g_3,
\end{equation*}
and  $\tilde{\wp}(x):=\wp(ix)$ satisfies the equation
\begin{equation*}
 \tilde{\wp}'(x)^2=-4\wp(z)^3+g_2\wp(z)+g_3
\end{equation*}
where $g_2, g_3$ are the Weierstrass invariants of $\wp$, which are real as $\wp$ is associated to the lattice $\langle1,i\tau\rangle$.
For the derivatives of the inverses of $\wp(x)$ and $\tilde{\wp}(x)$, we have the expressions
\begin{align*}
 \frac{d}{dx}\wp^{-1}(x)=\frac{1}{\wp'(\wp^{-1}(x))}=\frac{1}{\sqrt{4x^3-g_2x-g_3}},\\
 \frac{d}{dx}\tilde{\wp}^{-1}(x)=\frac{1}{\tilde{\wp}'(\tilde{\wp}^{-1}(x))}=\frac{1}{\sqrt{-4x^3+g_2x+g_3}}
\end{align*}
  where the sign of the square root depends on the branch of the inverse of $\wp$ or $\tilde{\wp}$ which is under consideration. Note that both expressions will be real in the domains under consideration.

 \begin{proposition}\label{realbound}
  Let $\wp$ have lattice $\langle1,i\tau\rangle$, where $\tau>0$ is real, and $P(X,Y)$ be a complex polynomial of total degree at most $d$. Then the number of zeros of $\text{\normalfont Im}(P(z,\wp(z)))$ on the open line $(\beta,\beta+\gamma)$, $\gamma\in\{1,i\tau\}$ between two adjacent poles of $\wp$ is bounded by $4d^2+6d+1$.
 \end{proposition}
 \begin{proof}
 
 By periodicity, $\wp(z)=\wp(z-\beta)$, so we may take a transformation of $z$ to consider the lines $(0,1)$, and $(0,i\tau)$. As $\wp$ has lattice $\langle1,i\tau\rangle$, and $\tau$ is real, it is real on these lines, and $\text{Im}(P(z,\wp(z))) = Q(x,\wp(\gamma x))$, $\gamma\in\{1,i\}$ for some real polynomial $Q(X,Y)$ of degree at most $d$. The argument proceeds identically for either line, so we consider $(0,1)$.
 
  Let $x_0\in(0,1)$ be such that $\wp'(x_0)=0$. Then $\wp(x)$ is $1$--$1$ on the interval $(0,x_0)$, so the number of zeros of $Q(x,\wp(x))$ is equal to that of $Q(\wp^{-1}(x),x)$ on the interval $(\wp(x_0),\infty)$. Considering the system
  \begin{align*}
   Q(u,x)&=0\\
   u-\wp^{-1}(x)&=0
  \end{align*}
  we have, letting $J$ be the Jacobian of the system $(Q,u-\wp^{-1}(x))$, by Theorem \ref{preimages}, that the number of nondegenerate solutions the system is bounded by an upper bound of the number of nondegenerate preimages of any point in the range of the system
  \begin{align*}
   Q(u,x)\\
   J(u,x).
  \end{align*}
 Letting $Q_X(X,Y) = \frac{\partial}{\partial X}Q(X,Y)$, and similarly for $Q_Y(X,Y)$, $J(x,u)$ is given by
  
  \begin{equation*}
   \frac{Q_X(u,x)}{\sqrt{4x^3-g_2x-g_3}}+Q_Y(u,x).
  \end{equation*}
  Taking some point $(a,b)$ in the range of the system $(Q,J)$, we bound the number of nondegenerate preimages. If $J(x,u)=b$, then
  \begin{equation*}
   J(u,x)=\frac{Q_X(u,x)}{\sqrt{4x^3-g_2x-g_3}}+Q_Y(u,x)=b,
  \end{equation*}
  and this holds iff
  \begin{equation*}
   Q_X(u,x)+(Q_Y(u,x)-b)\sqrt{4x^3-g_2x-g_3}=0,
  \end{equation*}
    and if this holds, then
    \begin{align*}
      (Q_X(u,x)&+(Q_Y(u,x)-b)\sqrt{4x^3-g_2x-g_3})\\
      &\cdot(Q_X(u,x)-(Q_Y(u,x)-b)\sqrt{4x^3-g_2x-g_3})\\
      &=Q_X(u,x)^2-(Q_Y(u,x)-b)^2(4x^3-g_2x-g_3)=0
    \end{align*}
    holds, so that the number of preimages of $(a,b)$ is bounded by the number of nondegenerate solutions of 
    \begin{align*}
     Q(u,x)-a&=0\\
     Q_X(u,x)^2-(Q_Y(u,x)-b)^2(4x^3-g_2x-g_3)&=0,
    \end{align*}
    which by B\'{e}zout's theorem is bounded by $2d^2+3d$.
    The interval $(x_0,1)$ is similar.

 \end{proof}

\begin{figure}
\centering
   \begin{tikzpicture}[scale=3]
      \draw[thick] (-0.2,0) -- (1.2,0);
      \draw[thick] (0,-0.2) -- (0,1.2);
      
      \draw[dashed] (-0.2,1) -- (1.2,1);
      \draw[dashed] (1,-0.2) -- (1,1.2);
      
      \draw[ultra thick][->] (0,0.2)--(0,0.5);
      \draw[ultra thick] (0,0.5)--(0,0.8);
      \draw[ultra thick] (0.2,0)--(0.8,0);
      \draw[ultra thick] (0.2,1)--(0.8,1);
      \draw[ultra thick] (1,0.2)--(1,0.8);
      \draw[ultra thick] (0,0.2) arc (90:0:0.2);
      \draw[ultra thick] (0,0.8) arc (-90:0:0.2);
      \draw[ultra thick] (0.8,1) arc (-180:-90:0.2);
      \draw[ultra thick] (1,0.2) arc (90:180:0.2);
    
   \end{tikzpicture}
   
   \caption{}
   \label{contour_2}
   
\end{figure}

 \begin{proof}[Proof of Theorem 2]
  Let $P(z):=P(z,\wp(z))$. We first take a box $B$ within the fundamental domain such that all zeros of $P$ in the interior of the fundamental domain lie within $B$. Next, define $\epsilon, \theta$ such that $0<\epsilon<\min_{\partial B}\lvert P\rvert$, and $P_\epsilon(z):=P(z)+\epsilon e^{i\theta}\neq0$ for $z \in \partial \mathcal{F}$. By Rouch\'{e}'s theorem, $P_\epsilon(z)$ has the same number of zeros in $B$ counting multiplicity as $P(z)$.
  
  We now bound the number of zeros of $P_\epsilon$ within the contour $\Gamma$ given by the truncations of the lines on the boundary of $\mathcal{F}$ united with interior quarter--circles about the poles of $\wp(z)$, as indicated in Figure \ref{contour_2} -- the common radii of these circles is taken so that there are no zeros of $P_\epsilon$ in a disc of this radius about the poles, and such that they do not intersect $B$. The radii will be taken sufficiently small subject to this. As $B$ lies in the interior of $\Gamma$, a bound upon the number of zeros of $P_\epsilon$ within $\Gamma$ bounds the number within $B$, and so bounds that of $P(z,\wp(z))$ within the fundamental domain.
  
  By the argument principle, the number of zeros of $P_\epsilon$ within $\Gamma$ is equal to 
  \begin{equation*}
   \frac{1}{2\pi i}\int_\Gamma\frac{P'_\epsilon(z)}{P_\epsilon(z)}dz,
  \end{equation*}
  which we now estimate.
  As $\wp(z)$ is real on the lines of $\Gamma$, considering for example the line $z=ix$, for $x$ real,
  \begin{equation*}
   \text{Re}(P_\epsilon(z))=\text{Re}(P(ix,\wp(ix)))+\epsilon\cos(\theta)=Q(x,\wp(ix)).
  \end{equation*}
  The number of zeros of which is, by Proposition [\ref{realbound}], bounded by $4d^2+6d+1$, and holds for any radius of quarter-circles. The other lines are similar, and so our bound for the integral over all of the lines is $8d^2+12d+6$ by Lemma \ref{integralzeros}.
  
  On the quarter--circles, we may expand $P_\epsilon$ into its Laurent series -- given a sufficiently small radius $\delta$, the term of smallest power will dominate. We let $\gamma$ be the path $p+\delta e^{i\xi}$, where $\xi$ ranges over the appropriate interval of length $\pi/2$ in $[0,2\pi]$, for the particular pole $p$, to be the interior quarter--circle. This path has length $\delta\pi/4$. Writing
  \begin{equation*}
   P_\epsilon(z) = a_k(z-p)^k +\sum_{n=k+1}^\infty a_n(z-p)^n,
   \end{equation*}
   we have, as the radius $\delta \to 0$,
    \begin{align*}
   \frac{P_\epsilon(z)-a_k(z-p)^k}{a_k(z-p)^k}&\to 0,\\
   \frac{\frac{d}{dz}(P_\epsilon(z)-a_k(z-p)^k)}{a_k(z-p)^k}&\to \frac{(k+1)a_{k+1}}{a_k}.
  \end{align*}
  So for sufficiently small $\delta$, we have, letting $f(z)=a_k(z-p)^k$, and $g(z)=P_\epsilon(z)-f(z)$, by Lemma \ref{integralremainder}, with $C=2$,
  \begin{align*}
   \frac{1}{2\pi}\left\lvert\int_\gamma \frac{P'_\epsilon(z)}{P_\epsilon(z)}dz\right\rvert&\leq \frac{1}{2\pi}\left\lvert\int_\gamma \frac{k}{z-p}dz\right\rvert+\frac{\delta}{2} \left(\frac{(k+1)a_{k+1}}{a_k}+0.1\right)\\
   &\leq \lvert k\rvert/4+0.1\\
   &\leq d/2+0.1,
  \end{align*}
  where the last inequality follows from the fact that  $\wp$ has poles of order $2$, so $-2d\leq k\leq d$. As there are $4$ quarter--circles about the poles of $\wp$, and $4$ line segments, the absolute value of the whole integral is bounded by $8d^2+14d+7$.
  \end{proof}


\begin{thebibliography}{1}

\bibitem{DLMF}
{NIST} digital library of mathematical functions.
\newblock F.~W.~J. Olver, A.~B. {Olde Daalhuis}, D.~W. Lozier, B.~I. Schneider,
  R.~F. Boisvert, C.~W. Clark, B.~R. Miller, B.~V. Saunders, H.~S. Cohl, and
  M.~A. McClain, eds.

\bibitem{Bianconi}
R.~Bianconi.
\newblock Some model theory of hypergeometric and {P}faffian functions.
\newblock {\em South Amer. J. Log.}, 2(2):297--318, 2016.

\bibitem{Binyamini}
G.~Binyamini.
\newblock Density of algebraic points on {N}oetherian varieties.
\newblock {\em Geom. Funct. Anal.}, 29(1):72--118, 2019.

\bibitem{ram}
S.~Cooper.
\newblock Inversion formulas for elliptic functions.
\newblock {\em Proc. Lond. Math. Soc. (3)}, 99(2):461--483, 2009.

\bibitem{gauss}
C.~F. Gauss.
\newblock Circa seriem infinitatam
  {$1+\frac{\alpha\beta}{1.\gamma}x+\frac{\alpha(\alpha+1)\beta(\beta+1)}{1.2.\gamma(\gamma+1)}xx+\frac{\alpha(\alpha+1)(\alpha+2)\beta(\beta+1)(\beta+2)}{1.2.3.\gamma(\gamma+1)(\gamma+2)x^3}+$}
  etc.
\newblock {\em Commentationes societatis regiae scientiarum Gottingensis
  recentiores}, II, 1813.

\bibitem{Jones}
G.~Jones and H.~Schmidt.
\newblock Pfaffian definitions of weierstrass elliptic functions.
\newblock {\em Mathematische Annalen}, 2020.

\bibitem{fewnomials}
A.~G. Khovanski\u{\i}.
\newblock {\em Fewnomials}, volume~88 of {\em Translations of Mathematical
  Monographs}.
\newblock American Mathematical Society, Providence, RI, 1991.
\newblock Translated from the Russian by Smilka Zdravkovska.

\bibitem{masser}
D.~Masser.
\newblock {\em Auxiliary Polynomials in Number Theory}.
\newblock Cambridge Tracts in Mathematics. Cambridge University Press, 2016.

\end{thebibliography}
\end{document}